\documentclass[12pt]{amsart}
\usepackage{amsfonts,latexsym,amscd}

\usepackage[all,cmtip,2cell]{xy}
\UseAllTwocells
\usepackage{verbatim}
\usepackage{hyperref}
\usepackage{mathrsfs}
\usepackage{amssymb}

\numberwithin{equation}{section}

\newtheorem{theorem}{Theorem}[section]
\newtheorem{lemma}[theorem]{Lemma}
\newtheorem{proposition}[theorem]{Proposition}
\newtheorem{corollary}[theorem]{Corollary}
\theoremstyle{definition}
\newtheorem{definition}[theorem]{Definition}

\newtheorem{example}[theorem]{Example}
\newtheorem{question}[theorem]{Question}

\newtheorem{remark}[theorem]{Remark}

\newcommand{\id}{{\rm Id}}

\newcommand{\Ker}{{\rm Ker\,}}

\newcommand{\Bnd}{{\rm End}}

\newcommand{\FPdim}{{\rm FPdim}}
\newcommand{\Bxt}{\text{\rm Ext}}

\newcommand{\Hom}{{\rm Hom}}

\newcommand{\Id}{{\rm Id}}

\newcommand{\Aut}{\text{Aut}}

\newcommand{\Rep}{{\rm Rep}}

\newcommand{\Vect}{{\rm Vec}}

\newcommand{\A}{\mathscr{A}}
\newcommand{\B}{\mathscr{B}}
\newcommand{\C}{\mathscr{C}}
\newcommand{\D}{\mathscr{D}}
\newcommand{\M}{\mathscr{M}}
\newcommand{\N}{\mathscr{N}}

\renewcommand{\O}{\mathscr{O}}

\newcommand{\g}{\mathfrak{g}}

\newcommand{\ot}{\otimes}

\newcommand{\ben}{\begin{enumerate}}
\newcommand{\een}{\end{enumerate}}

\newcommand{\Mod}{\mbox{Mod}}
\newcommand{\Bimod}{\mbox{Bimod}}

\begin{document}

\title[Exact factorizations and extensions] {Exact factorizations and extensions of finite tensor categories}

\author{Tathagata Basak}
\address{Department of Mathematics, Iowa State University, Ames, IA 50011, USA} \email{tathagat@iastate.edu}

\author{Shlomo Gelaki}
\address{Department of Mathematics, Iowa State University, Ames, IA 50011, USA} \email{gelaki@iastate.edu}

\date{\today}


\begin{abstract}
We extend \cite{G} to the nonsemisimple case. We define and study exact factorizations $\B=\A\bullet \C$ of a finite tensor category $\B$ into a product of two tensor subcategories $\A,\C\subset \B$, and relate exact factorizations of finite tensor categories to exact sequences of finite tensor categories with respect to exact module categories \cite{EG}. We apply our results to study exact factorizations of quasi-Hopf algebras, and extensions of a finite group scheme theoretical tensor category \cite{G2} by another one. We also provide several examples to illustrate our results. 
\end{abstract}

\maketitle

\tableofcontents

\section{Introduction}

Recall that a finite group $G$ admits an exact factorization $G=G_1G_2$ into a product of two subgroups $G_1,G_2\subset G$, if $G_1$ and $G_2$ intersect trivially and the order of $G$ is the product of the orders of $G_1$ and $G_2$. Equivalently, $G=G_1G_2$ is an exact factorization if every element $g\in G$ can be uniquely written in the form $g=g_1g_2$, where $g_1\in G_1$ and $g_2\in G_2$. Finite groups with exact factorization are fundamental objects in group theory, which naturally show up in many interesting results in the theories of Hopf algebras and tensor categories (see, e.g., \cite{BGM,K,M,Na,O}). 

In \cite{G} we introduced a categorical generalization of exact factorizations of finite groups for fusion categories $\B$ over $\mathbb{C}$. Namely, we 
say that $\B$ has an exact factorization $\B=\A\bullet \C$ into a product of two fusion subcategories $\A,\C\subset \B$, if $\FPdim(\B)= \FPdim(\A)\FPdim(\C)$ and $\A\cap \C=\Vect$. Then $\B=\A\bullet \C$ if and only if every simple object of $\B$ can be uniquely expressed in the form $X\ot Y$, where $X$ and $Y$ are simple objects of $\A$ and $\C$, respectively \cite[Theorem 3.8]{G}.
For example, exact factorizations $\B=\Vect(G_1)\bullet \Vect(G_2)$ are classified by pairs $(G,\omega)$, where $G$ is a group with exact 
factorization $G=G_1G_2$, and $\omega\in H^3(G,\mathbb{C}^{\times})$ that is trivial on $G_1$ and $G_2$.

We also related exact factorizations of fusion categories with exact sequences of fusion categories \cite[Theorem 4.1]{G}. Namely, let 
\begin{equation}\label{exseqfc}
\A\xrightarrow{\iota} \mathscr{B}\xrightarrow{F} \C\boxtimes \Bnd(\M)
\end{equation}
be an exact sequence of fusion categories  with respect to an exact indecomposable $\A$-module category $\mathscr{M}$ \cite{EG}. Then (\ref{exseqfc}) defines  
an exact factorization $\mathscr{B}^*_{\C\boxtimes \mathscr{M}}=\C\bullet \A^*_{\M}$, and vice versa, 
any exact factorization $\B=\A\bullet \C$ of fusion categories gives rise to an exact sequence of fusion categories with respect to any indecomposable $\A$-module category $\mathscr{M}$.

Our goal in this paper is to generalize the above mentioned results from \cite{G} to finite (not necessarily semisimple) tensor categories over an algebraically closed field of any characteristic $p\ge 0$. 

The structure of the paper is as follows. In \ref{prelim}, we recall some necessary background on exact module categories over finite tensor categories, exact sequences of finite tensor categories, and finite group scheme theoretical categories \cite{G2}. In \ref{efftcs}, we define and study factorizations $\B=\A\C$ and exact factorizations $\B=\A\bullet \C$ of a finite tensor category $\B$ into a product of two tensor subcategories $\A,\C\subset \B$. In \ref{efqhas}, we consider exact factorizations of quasi-Hopf algebra, which extends the notion of double cross coproduct Hopf algebras. 
In \ref{extftcs}, we prove Theorem \ref{mainnew}, which relates exact factorizations of finite tensor categories to exact sequences of finite tensor categories with respect to an indecomposable exact module. Finally, in \ref{extfgrsctcs}, we apply our results to finite group scheme theoretical categories. In particular, we show that any extension of a finite group scheme theoretical category by another one is {\em Morita equivalent} to a finite tensor category with an exact factorization into a product of two finite group scheme theoretical subcategories (Corollary \ref{gtext}). We also discuss some examples of exact factorizations and exact sequences (e.g., Corollary \ref{kacalg}). 

\section{Preliminaries}\label{prelim}
Throughout this paper, $k$ will denote an algebraically closed field of characteristic $p\ge 0$, and all categories are assumed to be $k$-linear abelian. We refer the reader to the book \cite{EGNO} as a general reference for the theory of finite tensor categories.

\subsection{Finite tensor categories}\label{Finite tensor categories} 
Let $\A$ be a finite tensor category over $k$, and let $\mathcal{O}(\A)$ denote the complete set of isomorphism classes of simple objects of $\A$. 
Let ${\rm Gr}(\A)$ be the Grothendieck ring of $\A$, and let ${\rm K_0}(\A)$ be
the group of isomorphism classes of projective objects in $\A$. Then ${\rm K_0}(\A)$ is a bimodule over ${\rm Gr}(\A)$, and we have a natural homomorphism (in general, neither surjective nor injective) $\tau_{\A}: {\rm K_0}(\A)\to {\rm Gr}(\A)$. 
Recall \cite[Subsection 2.4]{EO} that we have a character $\FPdim: {\rm Gr}(\A)\to \mathbb{R}$, attaching to $X\in \A$ the Frobenius-Perron dimension of $X$. 
Recall also that we have a virtual projective object $R_\A:=\sum_{X\in \mathcal{O}(\A)} \FPdim(X)P(X)$ in ${\rm K_0}(\A)\otimes_{\mathbb{Z}}\mathbb{R}$, where $P(X)$ denotes the projective cover of the simple object $X$, such that 
$$XR_\A=R_\A X=\FPdim(X) R_\A,\,\,\,X\in {\rm Gr}(\A).$$ Using $\tau_{\A}$, we may regard 
$R_\A$ as an element of ${\rm Gr}(\A)\otimes_{\mathbb{Z}}\mathbb{R}$. Following \cite[Subsection 2.4]{EO}, we set $\FPdim(\A):=\FPdim(R_\A)$.

For $X,Y,Z\in\mathcal{O}(\A)$, let 
$$N_{Y,Z}^X:=[Y\ot Z:X]=\dim\Hom(P(X),Y\ot Z).$$
Recall \cite[Proposition 6.1.2]{EGNO} that for any $X,Y \in\mathcal{O}(\A)$, we have  
\begin{equation}\label{projdec}
Y\ot P(X)=\bigoplus_{Z\in \mathcal{O}(\A)} N_{{}^*Y,Z}^X P(Z).
\end{equation}

Recall that a finite tensor category $\A$ is called {\em weakly-integral} if its FP dimension is an integer, {\em integral} if FP dimensions of objects are integers, {\em fusion} if it is semisimple, and {\em pointed} if every simple object is invertible. 
Recall also that $\A$ has the {\em Chevalley property} if the tensor product of semisimple objects is semisimple. In this case, the tensor subcategory of $\A$ generated by the simples of $\A$ is (the maximal) fusion subcategory of $\A$, which is denoted by $\A_{{\rm ss}}$.

\subsection{Exact module categories and algebras}\label{exmodcat}
Recall that a left $\A$-module category $\M$ is said to be {\em indecomposable} if it is not a direct sum of
two nonzero module categories.

Let $A$ be an algebra in $\A$, and let $\Mod(A)_{\A}$ denote the category of {\em right} $A$-modules in $\A$ \footnote{This notation is different than \cite{EGNO} where this would be denoted by $\Mod_{\A}(A)$.}. 
Recall that $A$ is called {\em exact} if $\Mod(A)_{\A}$ is an indecomposable exact $\A$-module category.

Let $\mathscr{M}$ be an indecomposable exact 
$\A$-module category\footnote{Here and below, by ``a module category'' we will mean a left module category, unless otherwise specified.} (such a category is always finite \cite[Lemma 3.4]{EO}). 
Let $\Bnd(\M)$ be the abelian category of {\em right exact} endofunctors of
$\mathscr{M}$, and let $\A_{\mathscr{M}}^*:=
\Bnd_{\A}(\mathscr{M})$ be the dual category of $\A$ with respect to $\M$, i.e., the category of  
$\A$-linear (necessarily exact \cite[Proposition 7.6.9]{EGNO}) endofunctors of $\mathscr{M}$. Recall that
composition of functors turns $\Bnd(\M)$ into a finite monoidal
category, and $\A_{\mathscr{M}}^*$ into a finite tensor category.

Let $\M$ be an exact left $\A$-module category, and let $\N$ be an exact right $\A$-module category. 
If $A_1$ and $A_2$ are algebras in $\A$ such that $\M={\rm Mod}(A_1)_{\A}$ and 
$\N={\rm Mod}_{\A}(A_2)$, then $\mathscr{N} \boxtimes_{\A}\mathscr{M}$
denotes the category of $(A_2,A_1)$-bimodules in $\A$, 
which can also be described as the category 
of left $A_2$-modules in $\M$, or the category 
of right $A_1$-modules in $\N$ (see \cite{DSS} and \cite[Section 7.8]{EGNO}).

\subsection{Exact sequences of finite tensor categories}\label{esftc} 
Let $\A,\B$ and $\C$ be finite tensor categories, 
let $\iota:\A\xrightarrow{1:1}\B$ be an injective tensor functor, and let $\M$ be an indecomposable exact module category over $\A$.
Let $F:\B\to \C\boxtimes \Bnd(\M)$ be a surjective (= dominant) exact monoidal functor such that 
$\iota(\A)=\Ker(F)$ (= the subcategory of $X\in \B$ such that $F(X)\in \Bnd(\M)$) \footnote{See \cite[Definition 1.8.3]{EGNO} for the definitions of injective and surjective functors.}.
Recall \cite{EG} that $F$ defines an {\em exact sequence of tensor categories
\begin{equation}\label{ES}
\A\xrightarrow{\iota} \mathscr{B}\xrightarrow{F} \C\boxtimes \Bnd(\M)
\end{equation}
with respect to $\mathscr{M}$} (= {\em $F$ is normal}),
if for every object $X\in \mathscr{B}$ there exists a subobject $X_0\subset X$ such that $F(X_0)$ 
is the largest subobject of $F(X)$ contained in $\Bnd(\M)\subset \C\boxtimes \Bnd(\M)$.  
In this case we will also say that $\B$ is an {\it extension of $\C$ by $\A$ with respect to $\M$}, and let $\Bxt(\C,\A,\M)$ denote the set of equivalence classes 
of exact sequences (\ref{ES}) with various functors $\iota,F$ \cite{EG}. 
For example, if $\B=\C\boxtimes \A$ and $\iota,F$ are the obvious functors, 
then $F$ defines an exact sequence with respect to any 
indecomposable exact $\A$-module category $\M$. 

By \cite[Theorem 3.4]{EG}, (\ref{ES}) defines an exact sequence of finite tensor categories if and only if $\FPdim(\B)=\FPdim(\A)\FPdim(\C)$.

Also, recall \cite[Theorem 3.6]{EG} that the natural functor 
\begin{equation}\label{equphistar}
\Phi_*: \mathscr{B}\boxtimes_{\A}\mathscr{M}\to \C\boxtimes \mathscr{M},
\end{equation} 
given by 
$$
\mathscr{B}\boxtimes_{\A}\mathscr{M}\xrightarrow{F\boxtimes_{\A} \id} \C\boxtimes \Bnd(\M)\boxtimes_{\A}\mathscr{M}=\C\boxtimes \mathscr{M}\boxtimes \A_{\mathscr{M}}^{*\rm op}\xrightarrow{\id\boxtimes \rho} \C\boxtimes \mathscr{M},
$$ 
is an equivalence of left $\B$-module and right $\A^{{*\rm op}}_{\M}$-module categories, where $\rho:\mathscr{M}\boxtimes \A_{\mathscr{M}}^{*\rm op} \to \mathscr{M}$ denotes the right action of $\A_{\mathscr{M}}^{*\rm op}$ on $\mathscr{M}$. (See \cite[Sections 2.4,2.5]{EG} and \cite{DSS} for more details.)

Finally, let $\mathscr{N}$ be an indecomposable exact module category over $\C$. 
Then $\mathscr{N}\boxtimes \mathscr{M}$ is an exact module category over 
$\C\boxtimes \Bnd(\M)$, and we have $(\C\boxtimes \Bnd(\M))_{\mathscr{N}\boxtimes \mathscr{M}}^*=\C_{\mathscr{N}}^*$. Then the dual sequence to (\ref{ES}), with respect to $\N\boxtimes \M$, 
\begin{equation}\label{ES1}
\A_{\mathscr{M}}^*\boxtimes \Bnd(\mathscr{N})\xleftarrow{\iota^*} \mathscr{B}_{\mathscr{N}\boxtimes \mathscr{M}}^*\xleftarrow{F^*} \C_{\mathscr{N}}^*
\end{equation}
is exact with respect to $\N$ \cite[Theorem 4.1]{EG}.

\subsection{Group scheme theoretical categories}\label{gscth} Let $G$ be a {\em finite} group scheme over $k$ (see, e.g., \cite[Section 2.1.1]{G2}), and let $\omega\in H^3(G,\mathbb{G}_m)$ be a normalized $3$-cocycle. That is, $\omega\in \mathscr{O}(G)^{\ot 3}$ is an
invertible element satisfying the associativity and 
normalization equations
$$({\rm Id}\ot {\rm Id}\ot \Delta)(\omega)(\Delta\ot {\rm Id}\ot {\rm Id})(\omega)=(1\ot \omega)({\rm Id}\ot \Delta\ot {\rm Id})(\omega)(\omega\ot 1),$$
$$(\varepsilon\ot {\rm Id}\ot {\rm Id})(\omega)=({\rm Id}\ot \varepsilon\ot {\rm Id})(\omega)=({\rm Id}\ot {\rm Id}\ot \varepsilon)(\omega)=1.$$
Equivalently, $\omega$ is a Drinfeld associator of the Hopf algebra $\O(G)$. 
Recall \cite[Section 3.1]{G2} that the category ${\rm Coh}(G)$, with tensor product given by convolution of sheaves and associativity constraint given by the action of $\omega$, is a finite tensor category denoted by ${\rm Coh}(G,\omega)$. Equivalently, ${\rm Coh}(G,\omega)$ is the tensor category ${\rm Rep}(\O(G),\omega)$ of finite dimensional representations of the quasi-Hopf algebra $(\O(G),\omega)$. (See \cite{EG2,G3} for  examples of nontrivial $3$-cocycles on non constant finite group schemes.)

Let $H\subset G$ be a closed subgroup scheme, and let $\psi\in C^2(H,\mathbb{G}_m)$ be a normalized $2$-cochain such that $d\psi=\omega_{|H}$. Let 
$$\M(H,\psi):={\rm Coh}^{(H,\psi)}(G,\omega)={\rm Mod}_{{\rm Coh}(G,\omega)}(k^{\psi}[H])$$ 
be the category of $(H,\psi)$-equivariant coherent sheaves on $(G,\omega)$ \cite[Section 3.2]{G2}. Then $\M(H,\psi)$ is an indecomposable exact module category over ${\rm Coh}(G,\omega)$ via convolution of sheaves, and the assignment $(H,\psi)\mapsto \M(H,\psi)$ is a bijection between conjugacy classes of pairs $(H,\psi)$ and equivalence classes of indecomposable exact module categories over
${\rm Coh}(G,\omega)$ \cite[Theorem 5.3]{G2}.

Let 
$$\C(G,\omega,H,\psi):={\rm Coh}(G,\omega)^*_{\M(H,\psi)}={\rm Bimod}_{{\rm Coh}(G,\omega)}(k^{\psi}[H]).$$  
A finite tensor category is called {\em group scheme theoretical} if it is tensor equivalent to some $\C(G,\omega,H,\psi)$ \cite[Definition 5.4]{G2}. 
For example, we have 
$$\C(G,\omega,1,1)={\rm Coh}(G,\omega),\,\,\,\,\,\C(G,1,G,1)={\rm Rep}(G).$$ 

By \cite[Theorem 5.7]{G2}, indecomposable exact module categories over $\C(G,\omega,H,\psi)$ are parametrized by conjugacy classes of pairs $(L,\eta)$, where $L\subset G$ is a closed subgroup scheme such that $\omega_{\mid L}=1$ and $\eta \in H^2(L,\mathbb{G}_m)$. Namely, the module category corresponding to a pair $(L,\eta)$ is the category $\N(L,\eta):={\rm Bimod}_{{\rm Coh}(G,\omega)}(k^{\psi}[H],k^{\eta}[L])$. We have,
\begin{equation}\label{dualgt}
\C(G,\omega,H,\psi)^*_{\N(L,\eta)}=\C(G,\omega,L,\eta).
\end{equation}
For example, $\Rep(G)^{*}_{\mathcal{N}(L,\eta)}=\C(G,1,L,\eta)$. Note that these are generalizations of Ostrik's results \cite{O} and 
reduce to them in the case of constant group schemes.

\subsection{Group scheme actions and equivariantization}\label{acgrsc}  
Retain the notation of \ref{gscth}. Let ${\rm Coh}_G$ denote the category of finite dimensional representations of $\O(G)$ endowed with the tensor product $\ot_{\O(G)}$. Then ${\rm Coh}_G$ is a monoidal category with unit object $\O(G)$.

Let $\C$ be a finite tensor category. 
Recall that the Deligne product ${\rm Coh}_G\boxtimes \C$ is a monoidal category, which can be interpreted as the category of $\O(G)$-modules in $\C$. Namely, the objects of ${\rm Coh}_G\boxtimes \C$ are pairs  
$(X,\phi)$, where $X\in\C$ and $\phi:\O(G)\to {\rm End}_{\C}(X)$ is an algebra map. A morphism $(X_1,\phi_1)\to (X_2,\phi_2)$ is a morphism $f:X_1\to X_2$ in $\C$ such that for every $\alpha\in\O(G)$, $\phi_2(\alpha)\circ f=f\circ \phi_1(\alpha)$ in ${\rm Hom}_{\C}(X_1,X_2)$.

We will denote ${\rm Coh}_G\boxtimes \C$ by ${\rm Coh}_G(\C)$ (so, ${\rm Coh}_G(\Vect)={\rm Coh}_G$). 

In particular, for every ${\rm V}=(V,{\rm m}_V)\in {\rm Coh}_G$ and $X\in\C$, we have the object ${\rm V}\boxtimes X:=(V\ot X,\phi)$ in ${\rm Coh}_G(\C)$, with $\phi$ being the composition
$$\O(G)\xrightarrow{{\rm m}_V} {\rm End}_{k}(V)\ot {\rm Id}_{X}\subset {\rm End}_{\C}(V\ot X).$$

Next recall that a $G$-action on $\C$ is a data $(T,\gamma)$, where 
$$T:\C\to {\rm Coh}_{\C}(G)$$ is a monoidal functor, and 
$$\gamma:({\rm Id}\boxtimes  T)\circ T\xrightarrow{\cong} ({\rm m}^*\boxtimes  {\rm Id})\circ T$$ is a functorial isomorphism preserving the monoidal structure in an appropriate sense, where ${\rm m}^*(V)=\mathscr{O}(G\times G)\ot _{\mathscr{O}(G)}V$ (see, e.g., \cite{DEN}). For example, the {\em trivial} $G$-action on $\C$ is given by $T_{{\rm triv}}(X)=\O(G)\boxtimes X$, $X\in\C$, and $\gamma_X=\iota\ot \Id_X$, where $\O(G)$ is viewed as the regular $\O(G)$-module and $\iota:\O(G)\ot \O(G)\xrightarrow{\cong}{\rm m}^*(\O(G))$ is the canonical isomorphism. 

Recall that the $G$-equivariantization tensor category $\C^G$ consists of pairs $(X,u)$, where $X\in \C$ and $u:T(X)\xrightarrow{\cong}T_{{\rm triv}}(X)=\O(G)\boxtimes X$,  
such that the diagram 
\begin{equation}\label{geqid}
\begin{CD}
(\Id\boxtimes T)(T(X))@>(\Id\boxtimes T)(u)>>(\Id\boxtimes T)(\O(G)\boxtimes X)=\O(G)\boxtimes T(X)\\
@V \gamma_X VV @V {\rm Id}_{\O(G)}\boxtimes u VV\\
({\rm m}^*\boxtimes {\rm Id})(T(X))@>({\rm m}^*\boxtimes {\rm Id})(u) >>({\rm m}^*\boxtimes {\rm Id})(\O(G)\boxtimes X)=\O(G)^{\boxtimes 2}\boxtimes X
\end{CD}
\end{equation}
commutes. The morphisms in $\C^G$ are the morphisms in $\C$ compatible with $u$. We have $\C^G$ is a finite tensor category.

Recall that ${\rm Rep}(G)$ can be identified with the tensor subcategory of $\C^G$ consisting of equivariant objects which are multiples of ${\bf 1}$ as objects of $\C$.

\begin{example}\label{repG}
Consider the trivial $G$-action $T_{{\rm triv}}$ on $\C:=\Vect$. An object $(X,u)$ in $\Vect ^G$ consists of a finite dimensional vector space $X$ and an $\O(G)$-linear isomorphism $u:\O(G)\ot X\xrightarrow{\cong}\O(G)\ot X$, such that diagram (\ref{geqid}) commutes. Since $\Hom_{\O(G)}(\O(G),\O(G))=\O(G)$, we have 
$$\Aut_{\O(G)}(\O(G),\O(G))=\O(G)^{\times}=\Hom(G,\mathbb{G}_m),$$
so we see that $u$ corresponds to an element of $\Hom_{{\rm Gr}}(G,{\rm GL}_k(X))$, i.e., to a representation of $G$ on $X$. 
Thus, $\Vect ^G\cong {\rm Rep}(G)$ as tensor categories.
\end{example}

\section{Exact factorizations of tensor categories}\label{efftcs}
In this section, $\B$ will denote a finite tensor category over $k$, $\A,\C$ will denote tensor subcategories of $\B$ \footnote{I.e., $\A,\C$ are full subcategories of $\B$, closed under taking subquotients, tensor products, and duality \cite[Definition 4.11.1]{EGNO}.}, and $\D:=\A\cap\C$ (see (\ref{Finite tensor categories})).

For $Y\in\mathcal{O}(\B)$, let $Y^D\in\mathcal{O}(\B)$ be such that $P(Y)^*=P(Y^D)$ \cite[Section 6.4]{EGNO}.  
Note that the natural surjective morphism \linebreak $\pi^{\B}_{Y^D}: P_{\B}(Y^D) \twoheadrightarrow Y^D$ induces an injective morphism 
\begin{equation*}
{}^*(Y^D) \xrightarrow{{}^*\pi^{\B}_{Y^D}} {}^*P_{\B}(Y^D) = {}^* ( P_{\B} (Y)^*) = P_{\B}(Y).
\end{equation*}

Let $X \in \mathcal{O}(\A)$, $Y \in \mathcal{O}(\C)$. If
$X={}^*(Y^D)$, then 
there is a nonzero morphism $P_{\A}(X) \to P_{\C}(Y)$ given by the composition
\begin{equation*}
P_{\A}(X) \xrightarrow{\pi^{\A}_{X}} X = {}^*(Y^D) 
\xrightarrow{{}^*\pi^{\C}_{Y^D}} P_{\C}(Y).
\end{equation*}
The lemma below says that if $\D$ is fusion then this is the only case when one has a nonzero morphism
in $\B$ from $P_{\A}(X)$ to $P_{\C}(Y)$. 

\begin{lemma}\label{factrorsgen}
For any $X\in\mathcal{O}(\A)$, $Y\in\mathcal{O}(\C)$ the following hold:
\begin{enumerate}
\item
If $\Hom_{\B}(P_{\A}(X),P_{\C}(Y))\ne 0$ then $X,\,{}^*(Y^D)$ belong to $\D$. 
\item 
If $\D$ is fusion, then $\dim\Hom_{\B}(P_{\A}(X),P_{\C}(Y))=\delta_{X^*,Y^D}$.
\end{enumerate}
\end{lemma}

\begin{proof}
(1) Let $0\ne f\in\Hom_{\B}(P_{\A}(X),P_{\C}(Y))$.
Let $Z$ be the image of $f$. Then $Z$ is a quotient (resp. subobject) 
of an object in $\A$ (resp. $\C$). So $Z$ belongs to
$\A \cap \C= \D$. 
One has nonzero surjections $P_{\A}(X) \twoheadrightarrow Z$ and
$P_{\C}(Y^D) \twoheadrightarrow Z^*$
(dual to the inclusion $Z \hookrightarrow P_{\C}(Y)$). 
Thus, $X$ is the unique simple quotient of $Z$ 
and $Y^D$  is the unique simple quotient $Z^*$, so 
${}^*(Y^D)$ is the unique simple subobject of $Z$.
So $X,{}^*(Y^D)\in \D$.

(2) Assume $\D$ is fusion and
$\Hom_{\B}(P_{\A}(X) , P_{\C}(Y)) \neq 0$.
Then the argument in (1) implies
$X = {}^*(Y^D)$.
Since $Z$ belongs to $\D$, 
it is a semisimple quotient of $P_{\A}(X)$.
But the unique semisimple quotient of $P_{\A}(X)$ is $X$,
so $Z=X$. Thus, any morphism $P_{\A}(X)$
to $P_{\C}(Y)$ factors through $X = {}^* (Y^D)$.
Hence,
$\Hom_{\B}(P_{\A}(X), X)$ is one dimensional \footnote{
Note that $\Hom_{\B}(P_{\A}(X), X) = \Hom_{\A}(P_{\A}(X), X)$
since $\A$ is a full subcategory of $\B$ and Hom space on the right hand side is one dimensional.},
and $\Hom_{\B}({}^*(Y^D), P_{\C}(Y))=\Hom_{\B}(P_{\C}(Y^D), Y^D)$
is one dimensional.
In other words, an arrow from $P_{\A}(X)$ to $X$, or from ${}^*(Y^D)$ to $P_{\C}(Y)$,
is unique up to a scalar.
Thus, any arrow from $P_{\A}(X)$ to $P_{\C}(Y)$ is unique up to a scalar as well.
\end{proof}

\begin{remark}
If $\D$ is not fusion, Lemma \ref{factrorsgen}(2) can fail. For example, take $\A=\B=\C=\Rep(\mathbb{Z}/p\mathbb{Z})$ in characteristic $p>0$. Then $P_{\A}({\bf 1})$ is the regular representation, so $\Hom_{\B}(P_{\A}({\bf 1}),P_{\C}({\bf 1}))=k[\mathbb{Z}/p\mathbb{Z}]$.
\end{remark}

\begin{corollary}\label{factrorsgen2}
Assume $\D=\Vect$. Then for every $X\in\mathcal{O}(\A)$ and $Y\in\mathcal{O}(\C)$, the following hold:
\begin{enumerate}
\item 
$\dim\Hom_{\B}(P_{\A}(X),P_{\C}(Y))=\delta_{X,{\bf 1}}\delta_{Y,{\bf 1}^D}$.
\item
For any $X'\in \mathcal{O}(\A)$, $Y'\in \mathcal{O}(\C)$ we have
$$\dim\Hom_{\B}(P_{\A}(X)\ot P_{\C}(Y),X'\ot Y')=\delta_{X,X'}\delta_{Y,Y'}$$
and
$$\dim\Hom_{\B}(X\ot Y,X'\ot Y')=\delta_{X,X'}\delta_{Y,Y'}.$$ 
\item
$X\ot Y$ is a brick \footnote{I.e., indecomposable, with $\Hom_{\B}(X\ot Y,X\ot Y)=k$.}
\end{enumerate}
\end{corollary}

\begin{proof}
(1) Lemma \ref{factrorsgen}(2) implies $X={}^*(Y^D)=\mathbf{1}$.
So, $Y^D = \mathbf{1}^* = \mathbf{1}$.
So, $Y^{**} = (Y^D)^D = \mathbf{1}^D$.
So, $Y = (\mathbf{1}^D)^{**} = \mathbf{1}^D$
since double dual of an invertible object is isomorphic to itself \cite[2.11.3]{EGNO}.

(2) By \cite[2.10.8]{EGNO}, we have $$\Hom_{\B}(P_{\A}(X)\ot P_{\C}(Y),X'\ot Y')=\Hom_{\B}(X'^*\ot P_{\A}(X),Y'\ot P_{\C}(Y)^*).$$
Since $X'^*\ot P_{\A}(X)$ and $Y'\ot P_{\C}(Y)^*$ are projective in $\A$ and $\C$, respectively, it follows from (\ref{projdec}) and (1) that 
\begin{eqnarray*}
\lefteqn{\dim\Hom_{\B}(X'^*\ot P_{\A}(X),Y'\ot P_{\C}(Y)^*)}\\
& = & \sum_{Z\in \mathcal{O}(\A),\,W\in \mathcal{O}(\C)} N_{X',Z}^X N_{{}^*Y',W}^{Y^D}\dim\Hom_{\B}(P_{\A}(Z),P_{\C}(W))\\
& = & N_{X',{\bf 1}}^X N_{{}^*Y',{\bf 1}^D}^{Y^D}=\delta_{X,X'}\delta_{Y,Y'},
\end{eqnarray*}
where the last equality follows from ${}^*Y'\ot {\bf 1}^D=(Y')^D$ \cite[Lemma 6.4.2]{EGNO}. This proves the first equality.

Since every morphism in $\Hom_{\B}(X'^*\ot X,Y'\ot Y^*)$ factors through its image, which is in 
$\A \cap \C=\Vect$, $\Hom_{\B}(X'^*\ot X,Y'\ot Y^*)\ne 0$ if and only if $X=X'$, $Y=Y'$. Moreover, since $\Hom_{\B}(P_{\A}(X)\ot P_{\C}(Y),X\ot Y)$ is spanned by the surjective homomorphism $\pi:=\pi^{\A}_{X}\ot \pi^{\C}_{Y}$,   it follows that for every $f$ in $\Hom_{\B}(X\ot Y,X\ot Y)$, $f\circ \pi$ and $\pi$ are proportional, which proves the second equality.

Finally by (2), $\Hom_{\B}(X\ot Y,X\ot Y)=k$, which implies (3).
\end{proof}

\begin{question}
In the setting of Corollary \ref{factrorsgen2}, is it true that $X\ot Y$ is simple? In the fusion case, the answer is yes \cite[Lemma 3.7]{G}.
\end{question}

Before we proceed, we will need the following simple lemma.

\begin{lemma}\label{basak1}
Let $W$ be a nonzero quotient of a projective object $P$ in $\B$, such that 
${\rm dim}{\rm Hom}(P, W) = 1$. Then $W$ is simple.
\end{lemma}

\begin{proof}
Let $f: P \twoheadrightarrow W$ be a nonzero surjection.
Writing $P$ as a direct sum of indecomposable projectives, 
and computing multiplicity with $W$, we find that
there exists a simple object $Z \in\mathcal{O}(\B)$ and an object $Q\in\B$, such that
$P = P(Z) \oplus Q$, ${\rm Hom}(Q, W) = 0$, 
and ${\rm dim}{\rm Hom}(P(Z), W) = 1$. It follows that $f$ factors through $P(Z)$ 
to give a surjection $P(Z) \twoheadrightarrow W$.

Let $Z' \in \mathcal{O}(\B)$.
The surjection $P(Z) \twoheadrightarrow W$ induces 
an injection ${\rm Hom}(W, Z') \hookrightarrow {\rm Hom}(P(Z), Z')$. 
It follows that ${\rm Hom}(W, Z') = 0$ for all $Z'\ne Z$, and 
${\rm Hom}(W, Z)$ is at most one dimensional. 
Since $W$ has a simple quotient, this simple quotient must be $Z$.
So there exists $P(W) \twoheadrightarrow P(Z)$.
So $P(W)= P(Z)$ and hence
$W =Z$ is simple.
\end{proof}

Let us now consider $\B$ as a left module category over $\A\boxtimes \C^{{\rm op}}$ in the natural way.

\begin{proposition}\label{simple}
If $\B$ is exact over $\A\boxtimes \C^{{\rm op}}$ and $\D=\Vect$, then for any $X\in \mathcal{O}(\A),Y\in \mathcal{O}(\C)$ the following hold:
\begin{enumerate}
\item  
$X\ot Y\in \mathcal{O}(\B)$, and it determines $X$, $Y$.
\item
$P_{\B}(X\otimes Y)=P_{\A}(X)\otimes P_{\C}(Y)$.
\end{enumerate} 
\end{proposition}

\begin{proof}
(1) Since $\B$ is exact over $\A\boxtimes \C^{{\rm op}}$,  
$P_{\A}(X)\otimes P_{\C}(Y)$ is projective in $\B$. 
Since $X\otimes Y$ is a nonzero quotient of $P_{\A}(X)\otimes P_{\C}(Y)$, with ${\rm Hom}(P_{\A}(X)\otimes P_{\C}(Y),X\otimes Y)$ one dimensional by Corollary \ref{factrorsgen2}(2), it follows from Lemma \ref{basak1} that $X\otimes Y$ is simple. The claim that $X\otimes Y$ determines $X,Y$ follows from the second equation in Corollary \ref{factrorsgen2}(2).
 
(2) By exactness, $P_{\A}(X)\otimes P_{\C}(Y)$ is a projective object in $\B$. 
Since any composition factor of $P_{\A}(X)\otimes P_{\C}(Y)$ is a composition factor of 
$X’ \otimes Y’$ for some composition factors
$X’$ of $P_{\A}(X)$ and $Y’$ of $P_{\C}(Y)$, it follows from (1) that all composition factors of $P_{\A}(X)\otimes P_{\C}(Y)$ are of the form $X'\otimes Y'$, where $X'\in \mathcal{O}(\A)$, $Y'\in \mathcal{O}(\C)$ are composition factors of $P_{\A}(X)$, $P_{\C}(Y)$, respectively. Hence, by  Corollary \ref{factrorsgen2}(2), $P_{\A}(X)\otimes P_{\C}(Y)$ is the projective cover of $X\ot Y$ in $\B$, as claimed. 
\end{proof}

If $\B$ is exact over $\A\boxtimes \C^{{\rm op}}$, we let $\A\C\subset \B$ denote the indecomposable exact $\A\boxtimes \C^{{\rm op}}$-module subcategory of $\B$ generated by ${\bf 1}$ (see \cite[Section 7.6]{EGNO}). Let 
$$R_{\A\C}:=\sum_{X\in \mathcal{O}(\A\C)}\FPdim(X)P_{\A\C}(X),$$
and set 
$$\FPdim(\A\C):=\sum_{X\in \mathcal{O}(\A\C)}\FPdim(X)\FPdim(P_{\A\C}(X)).$$

Recall that for every $Y\in \mathcal{O}(\C)$, $\FPdim(Y)=\FPdim(Y^D)$ \cite[Lemmas 6.4.1,6.4.2]{EGNO}.

\begin{proposition}\label{fpdimtriv} 
Assume $\B$ is exact over $\A\boxtimes \C^{{\rm op}}$, and $\D=\A\cap\C$ is fusion. Then 
$$\FPdim(\A)\FPdim(\C)=\FPdim(\A\C)\FPdim(\D).$$
\end{proposition}

\begin{proof}
Since $\A\C$ is an indecomposable exact $\A\boxtimes \C^{{\rm op}}$-module category, it follows from the Frobenius-Perron theorem that there is a unique (up to scaling) positive eigenvector for the action of $\A\boxtimes \C^{{\rm op}}$ on $\A\C$ (see \cite[Proposition 3.4.4]{EGNO}). Thus, we have $R_{\A}R_{\C}=\lambda R_{\A\C}$ for some scalar $\lambda$, since both $R_{\A\C}$ and $R_{\A}R_{\C}$ are positive eigenvectors for the action of $\A\boxtimes \C^{{\rm op}}$ on $\A\C$. 

Let us use the inner product given by $\dim\Hom(X,Y)$ to find $\lambda$ (see \ref{Finite tensor categories}). 
On the one hand, we have
$$\dim\Hom(R_{\A\C},{\bf 1})
=\sum_{X\in \mathcal{O}(\A\C)}\FPdim(X)\dim\Hom(P_{\A\C}(X),{\bf 1})=1.$$
On the other hand, using $R_{\C}^*=R_{\C}$, we have
\begin{eqnarray*}
\lefteqn{\dim\Hom(R_{\A}R_{\C},{\bf 1})=
\dim\Hom(R_{\A},R_{\C})}\\
& = & \sum_{X\in \mathcal{O}(\A),\,Y\in \mathcal{O}(\C)}\FPdim(X)\FPdim(Y)\dim\Hom(P_{\A}(X),P_{\C}(Y)).
\end{eqnarray*}
Thus by Lemma \ref{factrorsgen}(2), we have
\begin{eqnarray*}
\lefteqn{\dim\Hom(R_{\A}R_{\C},{\bf 1})}\\
& = & \sum_{Y\in \mathcal{O}(\D)}\FPdim({}^*(Y^D))\FPdim(Y)\dim\Hom(P_{\A}({}^*(Y^D)),P_{\C}(Y))\\
& = & \sum_{Y\in \mathcal{O}(\D)}\FPdim(Y)^2=\FPdim(\D).
\end{eqnarray*} 
Hence, $R_{\A}R_{\C}=\FPdim(\D)R_{\A\C}$, and taking $\FPdim$ of both sides, we get the statement. 
\end{proof}

\begin{corollary}\label{ineq}
Assume $\B$ is exact over $\A\boxtimes \C^{{\rm op}}$, and $\D=\A\cap\C$ is fusion. Then the following hold:
\begin{enumerate}
\item
$$\FPdim(\B)\ge \frac{\FPdim(\A)\FPdim(\C)}{\FPdim(\D)}.$$
\item 
We have an equality 
$$\FPdim(\B)=\frac{\FPdim(\A)\FPdim(\C)}{\FPdim(\D)}$$
if and only if $\B=\A\C$.
\end{enumerate}  
\end{corollary}

\begin{proof}
Since $\FPdim(\B)\ge \FPdim(\A\C)$, and we have an equality if and only if $\B=\A\C$, the corollary follows from Proposition \ref{fpdimtriv}. 
\end{proof}

\begin{definition}\label{facdef1}
We say that $\B$ factorizes into a product of $\A$ and $\C$ if $\B=\A\C$. A factorization $\B=\A\C$ is called {\em exact}, and denoted by $\B=\A\bullet \C$, if $\A\cap\C=\Vect$.
\end{definition}

\begin{remark}\label{permute}
Note that $\B=\A\C$ if and only if $\B=\C\A$, and $\B=\A\bullet\C$ if and only if $\B=\C\bullet\A$. 
\end{remark}

We are now ready to state and prove the main result of this section, which extends \cite[Theorem 3.8]{G} to the finite case.

\begin{theorem}\label{exfac}
Let $\B$ be a finite tensor category over $k$, and let $\A,\C$ be tensor subcategories of $\B$. Consider $\B$ as a left module 
category over $\A\boxtimes \C^{{\rm op}}$ in the natural way. 
The following are equivalent:
\begin{enumerate}
\item
For every $Z\in\mathcal{O}(\B)$, there exist unique $X\in \mathcal{O}(\A)$, $Y\in \mathcal{O}(\C)$ such that $Z=X\ot Y$ and $P_{\B}(Z)=P_{\A}(X)\otimes P_{\C}(Y)$.
\item
$\B=\A\bullet \C$.
\item
$\B$ is exact over $\A\boxtimes \C^{{\rm op}}$ and $\FPdim(\B)=\FPdim(\A)\FPdim(\C)$.
\end{enumerate}
\end{theorem}

\begin{proof}
Assume (1) holds. Then existence implies that any $Z\in\mathcal{O}(\B)$ is equivalent to ${\bf 1}$ for the equivalence relation defined 
in \cite[Section 7.6]{EGNO}, so $\B$ is indecomposable over $\A\boxtimes \C^{{\rm op}}$. Also, uniqueness implies $\A\cap\C=\Vect$. 

It remains to show that $\B$ is exact over $\A\boxtimes \C^{{\rm op}}$. To this end, 
note first that any projective in $\A\boxtimes \C^{{\rm op}}$ is a direct sum of projectives 
of the form $P_{\A}(X) \boxtimes P_{\C}(Y)$, $X\in \mathcal{O}(\A)$, $Y\in \mathcal{O}(\C)$. Now since tensoring preserves exactness in the 
$\A\boxtimes \C^{{\rm op}}$-module category $\B$, and an extension of a projective object by a projective 
one is projective, 
it suffices to show that $P_{\A}(X)\otimes Z\ot P_{\C}(Y)$ is projective in $\B$ for every $X\in \mathcal{O}(\A)$, $Y\in \mathcal{O}(\C)$ and $Z\in \mathcal{O}(\B)$ (and then induct on the length of an object of $\B$). Finally, to show this, note that by assumption, there are simples $X_1\in \mathcal{O}(\A)$, $Y_1\in \mathcal{O}(\C)$ such that $Z = X_1 \otimes Y_1$, so we have 
$$P_{\A}(X)\otimes Z\ot P_{\C}(Y)= 
P_{\A}(X)\otimes X_1 \otimes Y_1\ot P_{\C}(Y)
.$$ Since $P_{\A}(X)\otimes X_1$, $Y_1\ot P_{\C}(Y)$ are projective in $\A,\C$ respectively, it follows from (\ref{projdec}) 
that $P_{\A}(X)\otimes Z\ot P_{\C}(Y)$ is a direct sum of terms of the form $P_{\A}(X’) \otimes P_{\C}(Y’)$, where $X'\in \mathcal{O}(\A)$, $Y'\in \mathcal{O}(\C)$, which are all projective in $\B$ by assumption. Hence, (2) holds.

Conversely, assume (2) holds. Then by Proposition \ref{simple}, we have
\begin{eqnarray*}
\lefteqn{\FPdim(\B)=\sum_{Z\in \mathcal{O}(\B)}\FPdim(Z)\FPdim(P_{\B}(Z))}\\
& \ge & \sum_{\substack{X\in \mathcal{O}(\A),\,Y\in \mathcal{O}(\C)}}\FPdim(X\otimes Y)\FPdim(P_{\B}(X\otimes Y))\\
& = & \sum_{\substack{X\in \mathcal{O}(\A),\,Y\in \mathcal{O}(\C)}} \FPdim(X\otimes Y)\FPdim(P_{\A}(X)\otimes P_{\C}(Y))\\
& = & \FPdim(\A)\FPdim(\C).
\end{eqnarray*}
Since by Corollary \ref{ineq} we have an equality, (1) holds. 
 
Finally, (2) and (3) are equivalent by Corollary \ref{ineq}.
\end{proof}

\begin{question}
Let $\A,\C\subset \B$ be finite tensor categories. Is it true that if $\A\cap \C=\Vect$ and $\FPdim(\B)=\FPdim(\A)\FPdim(\C)$, then $\B$ is exact over $\A\boxtimes \C^{\rm {{\rm op}}}$? 
\end{question}

\begin{corollary}\label{pointedexfac}
Let $\B=\A\bullet \C$ be an exact factorization of finite tensor categories. Then the following hold:
\begin{enumerate}
\item
$\A,\C$ are weakly-integral if and only if $\B$ is weakly-integral.
\item
$\A,\C$ are integral if and only if $\B$ is integral.
\item
$\A,\C$ are fusion if and only if $\B$ is fusion.
\item
$\A,\C$ are pointed if and only if $\B$ is pointed.
\item
$\A,\C$ have the Chevalley property if and only if $\B$ has the Chevalley property. Moreover, in this case we have an exact factorization $\B_{{\rm ss}}=\A_{{\rm ss}}\bullet \C_{{\rm ss}}$ of fusion categories. (See \ref{Finite tensor categories}.)
\item
If $\B$ is braided then $\B=\A\boxtimes \C$ is a Deligne tensor product, and 
the simple objects of $\A$, $\C$ projectively centralize each other in the sense of \cite[Section 8.22]{EGNO}. 
\end{enumerate}
\end{corollary}

\begin{proof}
(1) Follows immediately from Theorem \ref{exfac} and \cite[Corollary 2.2]{GS}.

(2)-(4) Follow immediately from Theorem \ref{exfac}. 

(5) Let us prove the ``only if" part, the ``if" part being clear. 
Let $Z,Z'\in \mathcal{O}(\B)$. By Theorem \ref{exfac}, we have $Z=X\ot Y$ and $Z'=X'\ot Y'$ for unique $X,X'\in \mathcal{O}(\A)$ and $Y,Y'\in \mathcal{O}(\C)$. Now by Theorem \ref{exfac} again, $Y\ot X'$ is simple in $\B$, so $Y\ot X'=X''\ot Y''$ for unique simples $X''\in \mathcal{O}(\A)$ and $Y''\in \mathcal{O}(\C)$. Thus, we have
$$Z\ot Z'=(X\ot Y)\ot (X'\ot Y')=(X\ot X'')\ot (Y''\ot Y').$$
Since $\A,\C$ have the Chevalley property, $X\ot X''\in \A$ and $Y''\ot Y'\in\C$ are semisimple, hence $Z\ot Z'$ is semisimple by  Theorem \ref{exfac}.

(6) The proof is the same as the proof of \cite[Corollary 3.9]{G}, and we include it for the reader convenience.

The equivalence $F:\A\boxtimes \C\xrightarrow{\cong} \B$, $X\boxtimes Y\mapsto X\otimes Y$, has a tensor structure $J$ defined by the braiding structure ${\rm c}$ on $\B$. Namely, for every $X,X'\in \A$ and $Y,Y'\in \C$, $$J_{X\boxtimes Y,X'\boxtimes Y'}:F((X\boxtimes Y)\otimes (X'\boxtimes Y'))\xrightarrow{\cong} F(X\boxtimes Y)\otimes F(X'\boxtimes Y')$$ is the map
$$
\id\otimes {\rm c}_{X',Y}\otimes \id:(X\otimes X')\otimes (Y\otimes Y')\xrightarrow{\cong} (X\otimes Y)\otimes 
(X'\otimes Y').$$
It is straightforward to verify that since ${\rm c}$ satisfies the braiding axioms, $J$ satisfies the tensor functor axioms.

Furthermore, for every simple objects $X\in \mathcal{O}(\A)$ and $Y\in \mathcal{O}(\C)$, $${\rm c}_{Y,X}\circ {\rm c}_{X,Y}:X\otimes Y\xrightarrow{\cong}X\otimes Y$$ is an automorphism of the simple object $X\otimes Y$ in $\B$ (see Proposition \ref{simple}). Thus, ${\rm c}_{Y,X}\circ {\rm c}_{X,Y}=\lambda\cdot \id_{X\otimes Y}$ for some $\lambda\in k^{\times}$, as claimed.
\end{proof}

\section{Exact factorizations of quasi-Hopf algebras}\label{efqhas}

Recall that a finite dimensional Hopf algebra $B$ admits an exact factorization $B=AC$ as a product of its two Hopf subalgebras $A$ and $C$, if the multiplication map $A\ot C\to B$ is a (coalgebra) isomorphism, or equivalently, if the comultiplication map 
$$B^*\xrightarrow{\Delta} B^*\ot B^*\twoheadrightarrow A^*\ot C^*$$ is an (algebra) isomorphism \cite[p.316]{Maj}. By \cite[Theorem 7.2.3]{Maj}, $B=AC$ if and only if $B=A\bowtie C$ is a {\em double cross product} Hopf algebra (i.e., $A,C$ are a matched pair), if and only if, $B^*=A^*\bullet C^*$ is a {\em double cross coproduct} Hopf algebra of the quotient Hopf algebras $A^*,C^*$ of $B^*$ \cite[Exercise 7.2.15]{Maj}.

\begin{example}\label{exfacgrsch}
Let $H$ be any finite dimensional Hopf algebra, and let $D(H)$ be its Drinfeld double. Then we have $D(H)=H^{*{\rm cop}}\bowtie H$ and $D(H)^*=H^{{\rm op}}\bullet H^*$.
\end{example}

Let $(B,\Phi)$ be a finite dimensional {\em quasi}-Hopf algebra (see, e.g., \cite[Section 5.13]{EGNO}), and assume that
$$\pi_A:(B,\Phi)\twoheadrightarrow (A,\Phi_A),\,\,\,\,\,\,\pi_C:(B,\Phi)\twoheadrightarrow (C,\Phi_C)$$ are two surjective quasi-Hopf algebra maps.

\begin{definition}\label{exfacqhas}
We say that $(B,\Phi)$ admits an exact factorization $(B,\Phi)=(A,\Phi_A)\bullet(C,\Phi_C)$ if the composition map $$B\xrightarrow{\Delta} B\ot B\xrightarrow{\pi_A\ot \pi_C} A\ot C$$ is an (algebra) isomorphism.
\end{definition}

\begin{remark}
We have $(B,\Phi)=(A,\Phi_A)\bullet(C,\Phi_C)$ if and only if the composition map $A^*\ot C^*\xrightarrow{\Delta^*\circ(\pi_A^*\ot \pi_C^*)} B^*$ is a (coalgebra) isomorphism (because of finite dimensionality).
\end{remark}

\begin{theorem}\label{exfachas}
Let $(A,\Phi_A)$, $(C,\Phi_C)$ be two finite dimensional quasi-Hopf algebras, and set $\A:={\rm Rep}(A,\Phi_A)$ and $\C:={\rm Rep}(C,\Phi_C)$. Then exact factorizations $\B=\A\bullet \C$ are classified by quasi-Hopf algebras $(B,\Phi)$, with exact factorization 
$(B,\Phi)=(A,\Phi_A)\bullet (C,\Phi_C)$.
\end{theorem}

\begin{proof}
Assume that $(B,\Phi)=(A,\Phi_A)\bullet (C,\Phi_C)$. Then by definition, $B\cong A\ot C$ as algebras. Let $\B:={\rm Rep}(B,\Phi)$. Then $\A,\C\subset \B$ are tensor subcategories via $\pi_A$, $\pi_C$. Now a finite dimensional module $Z$ over $A\ot C$ is simple if and only if $Z=X\ot Y$ for unique finite dimensional simple $A$-module $X$ and $C $-module $Y$ (see, e.g, \cite{E}). Moreover, $P_{\A}(X)\otimes P_{\C}(Y)$ is indecomposable projective over $A\ot C$, and it has $X\ot Y$ as a simple quotient. Thus, $P_{\A}(X)\otimes P_{\C}(Y)$ is the projective cover of $X\ot Y$. This implies that 
$\A,\C\subset \B$ satisfy the conditions in Theorem \ref{exfac}(1), which proves that $\B=\A\bullet \C$, as required. 

Conversely, assume that $\B={\rm Rep}(A,\Phi_A)\bullet {\rm Rep}(C,\Phi_C)$. Then by Corollary \ref{pointedexfac}(2), $\B$ is integral. Hence by \cite[Proposition 2.6]{EO}, there exists a unique (up to twisting) quasi-Hopf algebra $(B,\Phi)$ such that $\B\cong {\rm Rep}(B,\Phi)$ as tensor categories. Since ${\rm Rep}(A,\Phi_A)$, ${\rm Rep}(C,\Phi_C)$ are tensor subcategories of $\B$, we have two surjective quasi-Hopf algebra homomorphisms 
$$\pi_A:(B,\Phi)\twoheadrightarrow (A,\Phi_A),\,\,\pi_C:(B,\Phi)\twoheadrightarrow (C,\Phi_C)$$
(using a standard argument as in, e.g., \cite[Proposition 2.2]{BoN}). 
Consider the representation of $B$ on $A\ot C$, induced by $(\pi_A\ot \pi_C)\circ\Delta$.
Since this representation is isomorphic to $$\left(\bigoplus_{X\in \mathcal{O}(A)}{\rm dim}(X)P(X)\right)\ot \left(\bigoplus_{Y\in \mathcal{O}(C)}{\rm dim}(Y)P(Y)\right),$$
it follows from Theorem \ref{exfac} that it is isomorphic to the regular representation $B$.
This implies that $(\pi_A\ot \pi_C)\circ\Delta:B\to A\ot C$ is bijective, as required.
\end{proof}

\begin{example}\label{classexfacrdb}
Let $B=A\bowtie C$ be any finite dimensional double cross product Hopf algebra. Then we have ${\rm Rep}(B^*)={\rm Rep}(A^*)\bullet {\rm Rep}(C^*)$. In particular, if $H$ is any finite dimensional Hopf algebra, then by Theorem \ref{exfachas}, we have ${\rm Rep}(D(H)^*)={\rm Rep}(H^{{\rm op}})\bullet {\rm Rep}(H^*)$ (see Example \ref{exfacgrsch}).
\end{example}

Let $G$ be a finite group scheme (see \ref{gscth}), and let $G_1,G_2\subset G$ be two closed subgroup schemes. 
Let us say that $G$ has an exact factorization $G=G_1G_2$ if $\mathscr{O}(G)=\mathscr{O}(G_1)\bullet\mathscr{O}(G_2)$.

\begin{corollary}\label{classexfacgrsch}
Let $G_1,G_2$ be two finite group schemes over $k$, and let $\omega_1\in H^3(G_1,\mathbb{G}_m)$, $\omega_2\in H^3(G_2,\mathbb{G}_m)$. Then exact factorizations $\B={\rm Coh}(G_1,\omega_1)\bullet {\rm Coh}(G_2,\omega_2)$   
are classified by pairs $(G,\omega)$, where $G$ is a finite group scheme with exact 
factorization $G=G_1G_2$, and $\omega\in H^3(G,\mathbb{G}_m)$ such that $\omega_1=\omega_{|G_1}$, $\omega_2=\omega_{|G_2}$. 
\end{corollary}

\begin{proof}
Assume $G=G_1G_2$, and $\omega\in H^3(G,\mathbb{G}_m)$. Let $\omega_1:=\omega_{\mid G_1}$ and $\omega_2:=\omega_{\mid G_2}$. Then we have $(\mathscr{O}(G),\omega)=(\mathscr{O}(G_1),\omega_1)\bullet (\mathscr{O}(G_2),\omega_2)$ by Theorem \ref{exfachas}, so ${\rm Coh}(G,\omega)={\rm Coh}(G_1,\omega_1)\bullet {\rm Coh}(G_2,\omega_2)$.

Conversely, assume $\B={\rm Coh}(G_1,\omega_1)\bullet {\rm Coh}(G_2,\omega_2)$. Let $A:=\mathscr{O}(G_1)$, $C:=\mathscr{O}(G_2)$. Then by Theorem \ref{exfachas}, there exists a quasi-Hopf algebra $(B,\Phi)$ such that $(\pi_A\ot \pi_C)\circ\Delta:B\xrightarrow{\cong} A\ot C$ is an algebra isomorphism, and $\B\cong {\rm Rep}(B,\Phi)$ as tensor categories. In particular, $B$ is a commutative Hopf algebra. Hence,  $B=\mathscr{O}(G)$ for some finite group scheme $G$, and $\Phi$ corresponds to a class $\omega\in H^3(G,\mathbb{G}_m)$ such that $\pi_A(\omega)=\omega_1$, $\pi_C(\omega)=\omega_2$. This proves the statement.
\end{proof}

Assume $k$ has characteristic $p>0$. Let $\g$ be a finite dimensional restricted Lie algebra over $k$, and let $\g_1,\g_2\subset\g$ be restricted Lie subalgebras. Let $G$, $G_1$ and $G_2$ be the finite group schemes with $\O(G)=u(\g)^*$, $\O(G_1)=u(\g_1)^*$ and $\O(G_2)=u(\g_2)^*$. (See, e.g., \cite[Section 2.2]{G2}.)
Let us say that $\g$ has an exact factorization $\g=\g_1\g_2$ if $\mathscr{O}(G)=\mathscr{O}(G_1)\bullet\mathscr{O}(G_2)$.

\begin{corollary}\label{classexfacrla}
Let $\g_1,\g_2$ be two finite dimensional restricted Lie algebras over $k$. Then exact factorizations $\B={\rm Rep}(u(\g_1)^*)\bullet {\rm Rep}(u(\g_2)^*)$ are classified by pairs $(\g,\omega)$, where $\g$ is a restricted Lie algebra with exact 
factorization $\g=\g_1\g_2$, and $\omega\in H^3(G,\mathbb{G}_a)=H^3(u(\g),k)$ is trivial on $G_1,G_2$ (equivalently, on $u(\g_1),u(\g_2)$).
\end{corollary}

\begin{proof}
Let $G_1,G_2$ be the finite group schemes with $\mathscr{O}(G_1)=u(\g_1)^*$ and $\mathscr{O}(G_2)=u(\g_2)^*$. By Corollary \ref{classexfacgrsch}, we have that exact factorizations $\B={\rm Rep}(u(\g_1)^*)\bullet {\rm Rep}(u(\g_2)^*)$ are classified by pairs $(G,\omega)$, where $G$ is a finite group scheme, with exact 
factorization $G=G_1G_2$ and $\omega$ in $H^3(G,\mathbb{G}_m)$ is trivial on $G_1$ and $G_2$. In particular, since $\mathscr{O}(G_1)$, $\mathscr{O}(G_2)$ are local Hopf algebras of height one, so is $\mathscr{O}(G)$ \footnote{I.e, $x^p=0$ for every $x$ in the maximal ideal of $\mathscr{O}(G)$.}. Thus, $\mathscr{O}(G)=u(\g)^*$ for some finite dimensional restricted Lie algebra $\g$. Finally, $H^3(G,\mathbb{G}_m)\cong H^3(G,\mathbb{G}_a)$ by \cite[Theorem 5.4]{EG2}.
\end{proof}

\section{Extensions of tensor categories}\label{extftcs}

Given two finite tensor categories $\A,\C$, let ${\rm ExFac}(\C,\A)$ denote the set of all triples $(\B,\iota_{\A},\iota_{\C})$ such that $\B$ is a finite tensor category, $\iota_{\A}:\A\xrightarrow{1:1} \B$ and $\iota_{\C}:\C\xrightarrow{1:1} \B$ are injective tensor functors, and $\B=\iota_{\A}(\A)\bullet \iota_{\C}(\C)$. 
Namely, $(\B,\iota_{\A},\iota_{\C})$ belongs to ${\rm ExFac}(\C,\A)$ if and only if $\A\boxtimes \C\xrightarrow{\ot\circ(\iota_{\A}\boxtimes\iota_{\C})}\B$ is an equivalence of abelian categories.

The following result relates exact factorizations and exact sequences of finite tensor categories (see \ref{esftc}). In particular, it extends \cite[Theorem 4.1]{G} to the finite case.

\begin{theorem}\label{mainnew}
Let $\A,\C$ be two finite tensor categories, and let $\M$ be any indecomposable exact $\A$-module category. The following hold:
\begin{enumerate}

\item
We have a map of sets
$$\alpha_{\M}:\Bxt(\C,\A,\M)\to {\rm ExFac}(\C^{\rm op},\A^*_{\M}),$$
given by $$\alpha_{\M}(\B,\iota_{\A},F)=(\mathscr{B}^*_{\C\boxtimes \mathscr{M}},\iota_{\A^*_{\M}},F^*_{\C\boxtimes \mathscr{M}}).$$

\item
We have a map of sets
$$\beta_{\M}:{\rm ExFac}(\C,\A) \to \Bxt(\C^{\rm op},\A^*_{\M},\M),$$
given by $$\beta_{\M}(\B,\iota_{\A},\iota_{\C})=(\B^*_{\B\boxtimes_{\A} \mathscr{M}},\iota_{\A^*_{\M}},(\iota_{\C})^*_{\B\boxtimes_{\A} \mathscr{M}}).$$
\end{enumerate} 
\end{theorem}

\begin{proof}
(1) Let 
\begin{equation}\label{startwithext}
\A\xrightarrow{\iota_{\A}} \mathscr{B}\xrightarrow{F} \C\boxtimes \Bnd(\M)
\end{equation}
be an exact sequence with respect to $\M$. Choose a generator $M$ for $\mathscr{M}$, and let $A:=\underline{\Bnd}_{\A}(M)$. Then $A$ is an indecomposable exact algebra in $\A$, and $\M\cong\Mod(A)_{\A}$ as $\A$-module categories (see \ref{exmodcat}). Since $A$ is also an indecomposable algebra in $\mathscr{B}$, it follows that  
the category $\Mod(A)_{\mathscr{B}}=\B\boxtimes_{\A}\M$ is an indecomposable $\B$-module category in the usual way. By \cite[Theorem 3.6]{EG}, we have an equivalence of $\B$-module categories $\B\boxtimes_{\A}\M\cong\C\boxtimes \mathscr{M}$, and  by \cite[Theorem 2.9]{EG}, $\C\boxtimes \mathscr{M}$ is exact over $\B$.  
Thus, $\mathscr{B}^*_{\C\boxtimes \mathscr{M}}\cong \B^*_{\Mod(A)_{\B}}\cong \Bimod_{\mathscr{B}}(A)^{\rm op}$ as tensor categories \cite[Remark 7.12.5]{EGNO}. Thus, we see that there is an inclusion of tensor categories
$\iota_{\A^*_{\M}}:\A^*_{\M}=\Bimod_{\A}(A)^{\rm op}\xrightarrow{1:1}\mathscr{B}^*_{\C\boxtimes \mathscr{M}}$.

Dualizing (\ref{startwithext}) with respect to $\C\boxtimes \mathscr{M}$, we get an exact sequence
\begin{equation}\label{exseqdual}
\C^{\rm op}\xrightarrow{F^*_{\C\boxtimes \mathscr{M}}} \mathscr{B}^*_{\C\boxtimes \mathscr{M}}\xrightarrow{(\iota_{\A})^*_{\C\boxtimes \mathscr{M}}} \A^*_{\M}\boxtimes \Bnd(\C)
\end{equation}
of finite tensor categories with respect to the indecomposable exact $\C^{\rm op}$-module category $\C$ (see (\ref{ES1})).  
In particular, $\C^{\rm op}$ can be identified with a tensor subcategory of $\mathscr{B}^{*}_{\C\boxtimes \mathscr{M}}$ via $F^*_{\C\boxtimes \mathscr{M}}$.

Now applying \cite[Theorem 3.6]{EG} (see (\ref{equphistar})) to the exact sequence (\ref{exseqdual}) provides a natural equivalence 
$$\Phi_*:\mathscr{B}^*_{\C\boxtimes \mathscr{M}}\xrightarrow{\cong} \A^*_{\M}\boxtimes \C^{\rm op}$$
of module categories over $\mathscr{B}^*_{\C\boxtimes \mathscr{M}}$. We claim that 
$$\Phi_*(X\ot Y)=X\boxtimes Y,\,\,\,X\in \A^*_{\M},\, Y\in \C^{\rm op}.$$ 
Indeed, note that an object $Y \in \C^{\rm op}$ acts on $\C \boxtimes \M$ via 

\begin{equation*}
Y_* := (\cdot \otimes Y) \boxtimes {\rm Id}_{\M} : \C \boxtimes \M \to \C \boxtimes \M.
\end{equation*}
This action is $\C \boxtimes {\rm End}(\M)$-linear, hence 
is also  $\B$-linear since the action of $\B$ on $\C \boxtimes \M$
factors through $\C \boxtimes {\rm End}(\M)$.
So, $F^*$ maps $Y$ to 
$Y_*  \in \B^*_{\C\boxtimes \M}$.

Also, note that an object $X \in \A^*_{\M}$ acts on $\B \boxtimes_{\A} \M$ via
\begin{equation*}
{\rm Id}_{\B} \boxtimes_{\A} X : \B \boxtimes_{\A} \M \to \B \boxtimes_{\A} \M,
\end{equation*}
and on $\C \boxtimes \M$ via
\begin{equation*}
X_* := {\rm Id}_{\C} \boxtimes X : \C \boxtimes \M \to \C \boxtimes \M.
\end{equation*}
Since the isomorphism $\Phi_*: \B \boxtimes_{\A} \M \xrightarrow{\cong} \C \boxtimes \M$ 
is right $\A^{{*\rm op}}_{\M}$-linear, we see that the two actions of $X$ on $ \B\boxtimes_{\A}\M$
and $\C\boxtimes\M$ given above correspond under $\Phi_*$.
Thus, $X \in \A^{{*\rm op}}_{\M}$ corresponds to 
$X_* \in \B^*_{\C \boxtimes \M}$.

It now follows from the above that 
\begin{equation*}
\Phi_*(X \otimes Y) = \Phi_*( X_* \circ Y_*) 
= \Phi_*( ( \cdot \otimes Y) \boxtimes X) = X \boxtimes Y,
\end{equation*}
as claimed.

Finally, it follows that for every simples $X\in \A^*_{\M}$, $Y\in \C$, $X\ot Y$ is simple in $\mathscr{B}^*_{\C\boxtimes \mathscr{M}}$, and every simple in $\mathscr{B}^*_{\C\boxtimes \mathscr{M}}$ is of this form. Moreover, since $\Phi_*(P_{\A^*_{\M}}(X)\ot P_{\C}(Y))=P_{\A^*_{\M}}(X)\boxtimes P_{\C}(Y)$ is the projective cover of $X\boxtimes Y$ in $\A^*_{\M}\boxtimes \C$, it follows that $P_{\A^*_{\M}}(X)\ot P_{\C}(Y)$ is the projective cover of $X\ot Y$ in $\mathscr{B}^*_{\C\boxtimes \mathscr{M}}$, so Theorem \ref{exfac} implies $\mathscr{B}^*_{\C\boxtimes \mathscr{M}}=\A^*_{\M}\bullet \C^{\rm op}$.

(2) Let $\B=\C\bullet \A$ be an exact factorization of finite tensor categories. Let $A\in\A$ be an exact algebra such that $\M=\Mod(A)_{\A}$. Since by \cite[Corollary 12.4]{EtO}, $A$ is also an exact algebra in $\B$, it follows that the category $\N:=\Mod(A)_{\B}$ is an indecomposable exact $\B$-module category. Since $\N=\B\boxtimes_{\A}\M=(\C\bullet \A)\boxtimes_{\A}\M$, it follows that $\N=\C\boxtimes \M$ as $\C$-module categories, so $\N$ is exact over $\C$. Note that we have $\C^*_{\N}=\C^*_{\C\boxtimes \M}=\C^{\rm op}\boxtimes \Bnd(\M)$.

Now, let $\A^*_{\M}$, $\B^*_{\N}$ be the dual finite tensor categories 
of $\A$, $\B$ with respect to $\M$, $\N$, respectively. Then we have a sequence
\begin{equation}\label{newexs}
\A^*_{\M}\xrightarrow{\iota_{\A^*_{\M}}} \B^*_{\N}\xrightarrow{(\iota_{\C})^*_{\N}} \C^{\rm op}\boxtimes \Bnd(\M).
\end{equation}
Note that $\A^*_{\M}$ is in the kernel of $(\iota_{\C})^*_{\N}$.  
Moreover, by Theorem \ref{exfac}, we have $\FPdim(\B)=\FPdim(\A)\FPdim(\C)$. Hence 
\begin{eqnarray*}
\lefteqn{\FPdim(\A^*_{\M})\FPdim(\C^{\rm op})=\FPdim(\A)\FPdim(\C)}\\
& = & \FPdim(\B)=\FPdim(\B^*_{\N}).
\end{eqnarray*}
Thus, by \cite[Theorem 3.4]{EG}, the sequence (\ref{newexs}) is exact with respect to $\M$, as claimed.
\end{proof}

\begin{corollary}\label{equiv} 
Let $\B=\C^{G}$ be a $G$-equivariantization of a finite tensor category $\C$ by a finite group scheme $G$ over $k$ (see \ref{acgrsc}). Then $\B^*_{\C}$ admits an exact factorization $\B^*_{\C}= \C^{{\rm op}}\bullet{\rm Coh}(G)$.
\end{corollary}

\begin{proof}
We have an exact sequence
$${\rm Rep}(G)\xrightarrow{\iota}\mathscr{B}\xrightarrow{F} \C$$ 
with respect to the standard fiber functor on ${\rm Rep}(G)$ (see Example \ref{repG}), 
so the claim follows from Theorem \ref{mainnew}.  
\end{proof}

\begin{corollary}\label{integral}
If $\A,\C$ are finite integral tensor categories, then any extension of $\C$ by $\A$ is integral.
 
Conversely, if some extension of $\C$ by $\A$ is integral then $\A,\C$ are integral.
\end{corollary}

\begin{proof}
Assume $\A,\C$ are integral, and $\B\in \Bxt(\C,\A,\M)$. By Theorem \ref{mainnew}, we have an exact factorization $\mathscr{B}^*_{\C\boxtimes \mathscr{M}}=\C^{\rm op}\bullet \A^*_{\M}$. Clearly, $\C^{\rm op}$ is integral. Also, since $\A$ is integral, it follows from \cite[Proposition 2.7]{EG} that FP dimensions of objects in $\A^*_{\M}$ are rational numbers. But FP dimensions are algebraic integers, so $\A^*_{\M}$ is integral too. Hence, $\mathscr{B}^*_{\C\boxtimes \mathscr{M}}$ is integral by Corollary \ref{pointedexfac}(2), so $\B$ is integral.

Conversely, assume $\B\in \Bxt(\C,\A,\M)$ is integral. By Theorem \ref{mainnew}, we have an exact factorization $\mathscr{B}^*_{\C\boxtimes \mathscr{M}}=\C^{\rm op}\bullet \A^*_{\M}$. Since $\B$ is integral, $\mathscr{B}^*_{\C\boxtimes \mathscr{M}}$ is integral, hence so are $\C^{\rm op}$ and $\A^*_{\M}$. Thus $\A,\C$ are integral.
\end{proof}

In Theorem \ref{mainnew}, setting $\alpha:=\alpha_{\A}$ and $\beta:=\beta_{\A}$, we have
\begin{equation*}\label{alphabeta}
\alpha(\B,\iota_{\A},F)=(\B^{\rm op},\iota_{\A^{\rm op}},F^*_{\C\boxtimes \A}),\,\,\,\beta(\B,\iota_{\A},\iota_{\C})=(\B^{\rm op},\iota_{\A^{\rm op}},(\iota_{\C})^*_{\B}).
\end{equation*}

\begin{corollary}\label{newcorollary}
The maps $\alpha,\beta$ are inverse to each other. Thus, we have a bijection of sets
$${\rm ExFac}(\C,\A)=\Bxt(\C,\A,\A).$$
\end{corollary}

\begin{proof}
We have  
\begin{eqnarray*}
\lefteqn{\alpha\beta(\B,\iota_{\A},\iota_{\C})}\\
& = & \alpha(\B^{\rm op},\iota_{\A^{\rm op}},(\iota_{\C})^*_{\B})=(\B,\iota_{\A},((\iota_{\C})^*_{\B})^*_{\C\boxtimes\A}))=(\B,\iota_{\A},((\iota_{\C})^*_{\B})^*_{\B}))\\
& = & (\B, \iota_{\A},\iota_{\C}),
\end{eqnarray*} 
and
\begin{eqnarray*}
\lefteqn{\beta\alpha(\B,\iota_{\A},F)}\\
& = & \beta(\B^{\rm op},\iota_{\A^{\rm op}},F^*_{\C\boxtimes \A})
=(\B,\iota_{\A},(F^*_{\C\boxtimes \A})^*_{\B^{\rm op}})=(\B,\iota_{\A},(F^*_{\B})^*_{\B^{\rm op}})\\
& = & (\B, \iota_{\A},F),
\end{eqnarray*}
as claimed.
\end{proof}

\begin{corollary}\label{chevalleyprop}
If $\A,\C$ are finite tensor categories with the Chevalley property, then any $\B\in \Bxt(\C,\A,\A)$ has the Chevalley property. 

Conversely, if some $\B\in \Bxt(\C,\A,\A)$ has the Chevalley property then $\A,\C$ have the Chevalley property.
\end{corollary}

\begin{proof}
By Corollary \ref{newcorollary}, we have $\B=\C\bullet \A$, so the claim follows from Corollary \ref{pointedexfac}(5).
\end{proof}

Let $A$ be a finite dimensional Hopf algebra. Recall that the forgetful functor ${\rm Rep}(A)\to {\rm Vec}$, equipped with the standard tensor structure, determines on ${\rm Vec}$ a structure of an exact indecomposable module category over ${\rm Rep}(A)$. Recall also, that $\Rep(A)^*_{\Vect}=\Rep(A^*)$ as tensor categories.

\begin{corollary}\label{mainnew2}
Let $A$ be a finite dimensional Hopf algebra, and let ${\rm Vec}$ be the standard module category over ${\rm Rep}(A)$ as above. Let $\C$ be any finite tensor category. Then the following hold:

\begin{enumerate}
\item
We have a map of sets
$$\alpha_{A}:\Bxt(\C,{\rm Rep}(A),{\rm Vec})\xrightarrow{\cong} {\rm ExFac}(\C^{\rm op},{\rm Rep}(A^*)),$$
given by $$\alpha_{A}(\B,\iota_{{\rm Rep}(A)},F)=\left(\B^*_{\C},\iota_{{\rm Rep}(A^*)},F^*_{\C}\right).$$
\item
We have a map of sets
$$\beta_{A}:{\rm ExFac}(\C,{\rm Rep}(A))\xrightarrow{\cong}\Bxt(\C^{\rm op},{\rm Rep}(A^*),{\rm Vec}),$$
given by $$\beta_{A}(\B,\iota_{_{{\rm Rep}(A)}},\iota_{\C})=\left(\B^*_{\C},\iota_{{\rm Rep}(A^*)},(\iota_{\C})^*_{\B\boxtimes_{\A} {\rm Vec}}\right).$$
\item
$\alpha_A$, $\beta_{A^*}$ are inverse to each other. Thus, $\alpha_A$, $\beta_{A}$ are bijective.
\end{enumerate} 
\end{corollary}

\begin{proof}
(1) and (2) follow from Theorem \ref{mainnew} since $\Rep(A)^*_{\Vect}=\Rep(A^*)$, and  
(3) follows since $(\B^*_{\M})^*_{\M}\cong \B$ as tensor categories.
\end{proof}

\begin{example}\label{equiv1}
Let $H$ be a finite dimensional {\em quasitriangular} Hopf algebra, and let $D(H)$ be its Drinfeld double. Then we have an exact sequence
$${\rm Rep}(H)\xrightarrow{\iota}{\rm Rep}(D(H))\xrightarrow{F} {\rm Rep}(H^{*{\rm cop}})$$ with respect to the standard fiber functor on ${\rm Rep}(H)$. Thus, by Corollary \ref{mainnew2}, we have an exact factorization 
$${\rm Bimod}_{{\rm Rep}(D(H))}(H^*)={\rm Rep}(D(H))^*_{{\rm Rep}(H^{*{\rm cop}})}={\rm Rep}(H^{*{\rm cop}})^{{\rm op}}\bullet {\rm Rep}(H^*).$$ 
\end{example}

\begin{example}\label{newex}
Let $G$ be a finite group scheme, let $H\subset G$ be a closed normal subgroup, and let $\omega\in H^3(G/H,\mathbb{G}_m)$.
We have an exact sequence
$${\rm Coh}(H)\xrightarrow{\iota}{\rm Coh}(G,\omega)\xrightarrow{F} {\rm Coh}(G/H,\omega)$$ of finite tensor categories with respect to the standard fiber functor on ${\rm Coh}(H)$. Since the ${\rm Coh}(G,\omega)$-module category ${\rm Coh}(G/H,\omega)$ (via $F$) is equivalent to $\mathscr{M}(H,1)$ (see \ref{gscth}), it follows from Corollary \ref{mainnew2} that we have an exact factorization $\C(G,\omega,H,1)={\rm Rep}(H)\bullet {\rm Coh}(G/H,\omega)$. (Note that this is a special case of Example \ref{main}(2), and of \cite[Proposition 3.1]{O} in the fusion case.)
\end{example}

\begin{corollary}\label{hopf}
Assume $\B={\rm Rep}(C)\bullet {\rm Rep}(A)$ is an exact factorization of finite tensor categories, where $A$ and $C$ are Hopf algebras. Then $\B$ is Morita equivalent to a tensor category with a fiber functor.
\end{corollary}

\begin{proof}
Follows immediately from Corollary \ref{mainnew2}.
\end{proof}

\begin{corollary}\label{fibf}
Assume $G=G_1G_2$ is an exact factorization of finite group schemes, and $\omega\in H^3(G,\mathbb{G}_m)$ such that $\omega_{\mid G_1}=\omega_{\mid G_2}=1$. Then $\C(G,\omega,G_2,1)$ admits a fiber functor.
\end{corollary}

\begin{proof}
By Corollary \ref{classexfacgrsch}, we have ${\rm Coh}(G,\omega)={\rm Coh}(G_1)\bullet {\rm Coh}(G_2)$. Let $\M:={\rm Mod}(k[G_2])_{{\rm Coh}(G_2)}=\Vect$ be the standard fiber functor on ${\rm Coh}(G_2)$, and let   
$$\N:={\rm Coh}(G,\omega)\boxtimes_{{\rm Coh}(G_2)}\M \cong {\rm Mod}(k[G_2])_{{\rm Coh}(G)}.$$ Then by Corollary \ref{mainnew2}, the finite tensor category
$${\rm Coh}(G,\omega)^*_{\N}\cong \C(G,\omega,G_2,1)$$
admits a fiber functor, as claimed.
\end{proof}
%

\section{Extensions of group scheme theoretical categories}\label{extfgrsctcs} 

Retain the notation from \ref{gscth}. The following result, whose proof follows immediately from Theorem \ref{mainnew}, extends \cite[Corollary 4.2]{G} to the finite case.

\begin{corollary}\label{gtext}
Any extension of a finite group scheme theoretical tensor category by another one is Morita equivalent to a tensor category with an exact factorization into a product of two group scheme theoretical subcategories. 

More precisely, 
let $\C(G_1,\omega_1,L_1,\psi_1)$ and $\C(G_2,\omega_2,L_2,\psi_2)$ be two finite group scheme theoretical categories, and let $\mathscr{N}=\mathscr{N}(L_3,\psi_3)$ be an indecomposable exact module category over $\C(G_1,\omega_1,L_1,\psi_1)$. Suppose 
$$\C(G_1,\omega_1,L_1,\psi_1)\xrightarrow{\iota}\mathscr{B}\xrightarrow{F} \C(G_2,\omega_2,L_2,\psi_2)\boxtimes \Bnd(\mathscr{N})$$ is an exact sequence with respect to $\mathscr{N}$. Then the dual tensor category $\mathscr{B}^*_{\C(G_2,\omega_2,L_2,\psi_2)\boxtimes \mathscr{N}}$ admits an exact factorization $$\mathscr{B}^*_{\C(G_2,\omega_2,L_2,\psi_2)\boxtimes \mathscr{N}}=\C(G_2,\omega_2,L_2,\psi_2)^{{\rm op}}\bullet   \C(G_1,\omega_1,L_3,\psi_3)$$
into a product of two group scheme theoretical subcategories. \qed
\end{corollary}

\begin{example}\label{main}
Let $G_1,G_2$ be finite group schemes. 

(1) Let $\mathscr{N}:=\mathscr{N}(L_3,\psi_3)$ be an indecomposable exact module category over $\Rep(G_1)$, and suppose $$\Rep(G_1)\xrightarrow{\iota}\mathscr{B}\xrightarrow{F} {\rm Coh}(G_2,\omega_2)\boxtimes \Bnd(\mathscr{N})$$ is an exact sequence with respect to $\mathscr{N}$. Then we have 
$$\mathscr{B}^*_{{\rm Coh}(G_2,\omega_2)\boxtimes \mathscr{N}}={\rm Coh}(G_2,\omega_2)^{{\rm op}}\bullet\C(G_1,1,L_3,\psi_3).$$

(2) Let $\M:=\M(L_3,\psi_3)$ be an indecomposable exact module category over ${\rm Coh}(G_1,\omega_1)$, and suppose  
$$
{\rm Coh}(G_1,\omega_1)\xrightarrow{\iota} \mathscr{B}\xrightarrow{F} {\rm Coh}(G_2,\omega_2)\boxtimes \Bnd(\M)$$ 
is an exact sequence with respect to $\M$. 
Then we have 
$$\mathscr{B}^*_{{\rm Coh}(G_2,\omega_2)\boxtimes \M}={\rm Coh}(G_2,\omega_2)^{{\rm op}}\bullet\C(G_1,\omega_1,L_3,\psi_3).$$
\end{example}

\begin{remark}\label{Nikshych}
In Example \ref{main} one does not always get a group scheme theoretical category. In particular, an exact factorization with two group scheme theoretical factors is not always group scheme theoretical. 
(See \cite[Corollary 4.6]{Ni}, \cite{G}.)
\end{remark}

\begin{corollary}\label{kacalg}
Let $G_1,G_2$ be finite group schemes. The following hold:
\begin{enumerate}
\item
If 
$${\rm Rep}(G_1)\xrightarrow{\iota}\mathscr{B}\xrightarrow{F} {\rm Coh}(G_2,\omega_2)$$ is an exact sequence with respect to the standard fiber functor $\N(1,1)$ on ${\rm Rep}(G_1)$, then 
$$\mathscr{B}\cong\C(G,\omega,G_2,1)$$  
for a finite group scheme $G$, with an exact factorization \linebreak $G=G_1G_2$ and $\omega\in H^3(G,\mathbb{G}_m)$ such that $\omega_{\mid G_1}=1$, $\omega_{\mid G_2}=\omega_2$.
\item
If 
$${\rm Coh}(G_1,\omega_1)\xrightarrow{\iota} \mathscr{B}\xrightarrow{F}{\rm Coh}(G_2,\omega_2)\boxtimes \Bnd(\mathscr{M}(1,1))$$ is an exact sequence with respect to $\mathscr{M}(1,1)={\rm Coh}(G_1,\omega_1)$, then 
$$\mathcal{B}\cong\C(G,\omega,G_2,1)$$ for a finite group scheme $G$, with an exact factorization \linebreak $G=G_1G_2$ and    
$\omega\in H^3(G,\mathbb{G}_m)$ such that $\omega_{\mid G_1}=\omega_1$, $\omega_{\mid G_2}=\omega_2$.
\end{enumerate}
\end{corollary}

\begin{proof}
(1) Specializing Example \ref{main}(1) to $\N(1,1)$ yields $$\mathscr{B}^*_{{\rm Coh}(G_2,\omega_2)}={\rm Coh}(G_1)\bullet {\rm Coh}(G_2,\omega_2).$$ 
Thus by Corollary \ref{classexfacgrsch}, we have $\mathscr{B}^*_{{\rm Coh}(G_2,\omega_2)}={\rm Coh}(G,\omega)$ for a finite group scheme $G$, with an exact factorization $G=G_1G_2$ and $\omega\in H^3(G,\mathbb{G}_m)$ such that $\omega_{\mid G_1}=1$ and $\omega_{\mid G_2}=\omega_2$. Thus, we have $\B={\rm Coh}(G,\omega)^*_{{\rm Coh}(G_2,\omega_2)}=\C(G,\omega,G_2,1)$, as claimed. 

(2) Specializing Example \ref{main}(2) to $\M(1,1)$ yields $$\mathscr{B}^*_{{\rm Coh}(G_2,\omega_2)}={\rm Coh}(G_1,\omega_1)\bullet {\rm Coh}(G_2,\omega_2).$$ Thus by Corollary \ref{classexfacgrsch}, we have $\mathscr{B}^*_{{\rm Coh}(G_2,\omega_2)}={\rm Coh}(G,\omega)$ as in (1). Hence $\B={\rm Coh}(G,\omega)^*_{{\rm Coh}(G_2,\omega_2)}=\C(G,\omega,G_2,1)$, as claimed.
\end{proof}

\begin{corollary}\label{nice}
Let $G_1,G_2$ be finite group schemes, with $G_1$ being commutative. Then equivalence classes of exact sequences
$${\rm Coh}(G_1)\xrightarrow{\iota}\mathscr{B}\xrightarrow{F} {\rm Coh}(G_2)$$ of finite tensor categories with respect to 
the standard fiber functor on ${\rm Coh}(G_1)$ are classified by conjugacy classes of pairs $(G,\omega)$, where $G$ is a finite group scheme with an exact factorization $G=G_1^D G_2$ \footnote{Here, $G_1^D$ denotes the Cartier dual of $G_1$.}, and    
$\omega\in H^3(G,\mathbb{G}_m)$ such that $\omega_{\mid G_1^D}=\omega_{\mid G_2}=1$.
\end{corollary}

\begin{proof}
Since $G_1$ is commutative, ${\rm Coh}(G_1)={\rm Rep}(G_1^D)$. Thus, by  
Corollary \ref{kacalg}(1), the assignment $(G,\omega)\mapsto \B=\C(G,\omega,G_2,1)$ is bijective.
\end{proof}

\begin{example}\label{kacalg2}
Let $G=G_1G_2$ be an exact factorization of finite group schemes. Assume $B$
is a Hopf algebra fitting into an exact sequence of Hopf algebras
$$\mathscr{O}(G_2)\to B\to k[G_1].$$ 
Let $\B:=\Rep(B)$. Then we have an exact sequence $$\Rep(G_1)\xrightarrow{\iota}\mathscr{B}\xrightarrow{F} {\rm Coh}(G_2)$$ with respect to the standard fiber functor on $\Rep(G_1)$. Hence,  
by Corollary \ref{kacalg}(1), $\mathscr{B}=\C(G,\omega,G_1,1)$ for some $3$-cocycle $\omega\in H^3(G,\mathbb{G}_m)$ that is trivial on $G_1$ and $G_2$ (but not necessarily on $G$). Thus, $B$ is group scheme theoretical.

Conversely, Corollary \ref{kacalg}(1) says that if $\B$ fits into an exact sequence
$${\rm Rep}(G_1)\xrightarrow{\iota}\mathscr{B}\xrightarrow{F} {\rm Coh}(G_2,\omega_2)$$ of finite tensor categories with respect to the standard fiber functor on ${\rm Rep}(G_1)$, then $\mathscr{B}=\C(G,\omega,G_1,1)$ is the representation category of a finite dimensional group scheme theoretical quasi-Hopf algebra.
\end{example}


\begin{thebibliography}{ABCD}

\bibitem
[BoN]{BoN} C-G. Bontea and D. Nikshych. On the Brauer-Picard group of a finite symmetric tensor category. {\em J. Algebra} {\bf 440} (2015), 
187--218.

\bibitem
[BrN1]{BN} A. Brugui\`eres and S. Natale. Exact sequences of tensor categories. {\em Int. Math. Res. Not.} {\bf 24} (2011), 5644--5705.

\bibitem
[BrN2]{BN1} A. Brugui\`eres and S. Natale. Central exact sequences of tensor categories, equivariantization and applications. {\em J. Math. Soc. Japan} {\bf 66} (2014), no. 1, 257--287.

\bibitem
[BGM]{BGM} E. Beggs, J. Gould and S. Majid. Finite group factorizations and braiding. \textit{J. Algebra} \textbf{181} (1996), no. 1, 112--151.

\bibitem
[DEN]{DEN} A. Davydov, P. Etingof and D. Nikshych. 
Autoequivalences of tensor categories attached to quantum groups at roots of $1$. {\em arXiv:1703.06543}.

\bibitem
[DSS]{DSS} C. Douglas, C. Schommer-Pries and N. Snyder. 
The balanced tensor product of module categories. {\em arXiv:1406.4204}.

\bibitem
[E]{E} P. Etingof et al.
Introduction to representation theory. {\em Student Mathematical Library} {\bf 59} (2011), 228 pp.

\bibitem
[EG1]{EG} P. Etingof and S. Gelaki. Exact sequences of tensor categories with respect to a module category. {\em Advances in Mathematics} {\bf 308} (2017), 1187--1208.

\bibitem
[EG2]{EG2} P. Etingof and S. Gelaki. Finite symmetric integral tensor categories with the Chevalley property, with an Appendix by Kevin Coulembier and Pavel Etingof. {\em International
Mathematics Research Notices} {\bf 12} (2019), 9083--9121.

\bibitem
[EO1]{EO} P. Etingof and V. Ostrik. Finite tensor categories. \textit{Moscow Math.\ Journal} \textbf{4} (2004), 627--654.

\bibitem
[EO2]{EtO} P. Etingof and V. Ostrik. On the Frobenius functor for symmetric tensor categories in positive characteristic. 
{\em arXiv:1912.12947}.

\bibitem
[ENO]{ENO2} 
P. Etingof, D. Nikshych and V. Ostrik. Weakly group-theoretical and solvable fusion categories. 
{\em Advances in Mathematics} {\bf 226} (2011), no. 1, 176--205.

\bibitem
[EGNO]{EGNO} P. Etingof, S. Gelaki, D. Nikshych and V. Ostrik. Tensor categories. {\em AMS Mathematical Surveys and Monographs book series} {\bf 205} (2015), 362 pp.

\bibitem
[G1]{G} S. Gelaki. Exact factorizations and extensions of fusion categories. {\em Journal of Algebra} {\bf 480} (2017), 505--518.

\bibitem
[G2]{G2} S. Gelaki. Module categories over affine group schemes. {\em Quantum Topology} {\bf 6} (2015), no. 1, 1--37.

\bibitem
[G3]{G3} S. Gelaki. 
Minimal extensions of Tannakian categories in positive characteristic. {\em arXiv:2105.13436}.

\bibitem
[GS]{GS} S. Gelaki and D. Sebbag. On finite non-degenerate braided tensor categories with a Lagrangian subcategory. {\em Transactions of the AMS}. {\em arXiv:1835.01568}.

\bibitem
[K]{K} G. I. Kac. Extensions of groups to ring groups. \textit{Math. USSR. Sb.} \textbf{5} (1968), 451--474.

\bibitem
[Maj]{Maj} S. Majid. Foundations of quantum group theory. {\em Cambridge University Press} (1995), 607 pp. 

\bibitem
[Mas]{M} A. Masuoka. Extensions of Hopf algebras. \textit{Trabajos de Matematica} \textbf{41/99}, Fa.M.A.F. (1999).

\bibitem
[Na]{Na} S. Natale. On group theoretical Hopf algebras and exact factorizations of finite groups. \textit{J.\ Algebra} \textbf{270} (2003), 199--211.

\bibitem
[Ni]{Ni} D. Nikshych. Non-group-theoretical semisimple Hopf algebras from group actions on fusion categories. \textit{Sel. math., New ser.} \textbf{14} (2008), 145--161.

\bibitem
[O]{O} V. Ostrik. Module categories over the Drinfeld double of a finite group. 
\textit{Int.\  Math.\ Res.\ Not.} (2003), 1507--1520.

\bibitem
[S]{S} P. Schauenburg. Hopf bimodules, coquasi-bialgebras, and an exact sequence of Kac. \textit{Advances in Math.} \textbf{165} (2002), 194--263.

\end{thebibliography}
\end{document}